\documentclass[11pt]{article}
\usepackage{mathrsfs}
\usepackage{amsmath, mathtools}
\usepackage{amssymb}
\usepackage{amsthm}
\usepackage{graphicx}
\usepackage{epic}
\usepackage{xcolor,cite}
\usepackage{cases}
\usepackage{pst-poly}  
\usepackage{pst-plot}  
\usepackage{pst-poly}  
\usepackage{CJK,makecell}

\usepackage{verbatim}
\usepackage{subfig}
\usepackage{booktabs}
\definecolor{dkgreen}{rgb}{0,0.6,0}
\numberwithin{equation}{section}

\usepackage{etex}

\usepackage[left,mathlines,displaymath]{lineno}

\usepackage{pgf}
\renewcommand{\paragraph}{\roman{paragraph}}

\usepackage{bm}
\usepackage[hidelinks]{hyperref}
\usepackage{tikz}
\usetikzlibrary{automata}
\usepackage{enumitem}
\usetikzlibrary{calc}
\usetikzlibrary{arrows,shapes,positioning}
\usetikzlibrary{decorations.markings}
\tikzstyle arrowstyle=[scale=1]
\tikzstyle directed=[postaction={decorate,decoration={markings, mark=at position .65 with {\arrow[arrowstyle]{stealth}}}}]
\tikzstyle reverse directed=[postaction={decorate,decoration={markings, mark=at position .65 with {\arrowreversed[arrowstyle]{stealth};}}}]

\newtheorem{theorem}{Theorem}[section]

\newtheorem{definition}[theorem]{Definition}

\newtheorem{problem}[theorem]{Problem}

\newtheorem{example}[theorem]{Example}

\newtheorem{Remark}[theorem]{Remark}

\begin{document}

\title{The enumeration of odd spanning trees in graphs\thanks{E-mail addresses: xush0928@163.com (S.\ Xu), xukx1005@nuaa.edu.cn (K.\ Xu). $\dag$ Corresponding author.}}\author{{Shaohan Xu, Kexiang Xu $^{\dag}$}\\\\{\small School of Mathematics, Nanjing University of Aeronautics and Astronautics,}\\
{\small Nanjing, Jiangsu 211106, PR China}\\
}

\maketitle
\begin{abstract}
A graph  is  odd if  all of its vertices  have odd degrees. In particular, an odd spanning tree  in a connected graph  is a spanning tree in which all vertices have  odd degrees. In this paper we establish a unified technique to  enumerate odd spanning trees of a graph $G$ in terms of a  multivariable polynomial associated with $G$ and indeterminates  $\{x_{i}:v_i\in V(G)\}$. As applications, the enumerative formulas for odd spanning trees in complete graphs, complete multipartite graphs,  almost complete graphs, complete split graphs and Ferrers graphs are,  respectively, derived from our work.
\end{abstract}

\section{Introduction}
 All graphs considered in this paper are finite, undirected and without loops.  Let $G$ be a weighted  graph with vertex set $V(G)$ and edge set $E(G)$, where each edge $e=ij\in E(G)$ is weighted by an indeterminate $\omega_e(G)=\omega_{ij}(G)$. The word ``weighted" is omitted when $\omega_{e}(G)=1$ for each $e\in E(G)$.  A spanning tree of $G$ is a connected spanning subgraph containing no cycles.  Denote by $\mathcal{T}(G)$  the set of spanning trees of $G$.   Let $\tau(G,\omega)$ be the {\it sum of weights of spanning trees} of $G$, where the weight of a spanning tree $T$ in $G$ is the product of weights of edges in $T$. That is,
 \begin{equation}\label{x1-3}
 \tau(G,\omega)=\sum_{T\in \mathcal{T}(G)}\prod_{e\in E(T)}\omega_{e}(G),
 \end{equation}
 which is also called the {\it Kirchhoff polynomial} \cite{C2} or  the {\it weighted spanning tree enumerator}  \cite{M2}  of $G$.
 If $\omega_{e}(G)=1$ for each edge $e\in E(G)$, then $\tau(G,\omega)=|\mathcal{T}(G)|$ is the number of spanning trees in $G$, and we write $\tau(G,\omega)$ simply as  $\tau(G)$ in this case.

Let $G$ be a graph of order $n$ and  $\{x_i:v_i\in V(G)\}$ be  indeterminates   on $V(G)$. If each edge $v_iv_j$ in $G$ has  weight $\omega_{v_iv_j}(G)=x_{i}x_{j}$, then $G$ is a {\it weighted graph induced by its vertex
weights} or {\it the weights of $G$ are induced by} $\{x_{i}:v_i\in V(G)\}$.  For the  graph $G$  with $v\in V(G)$, let $N_{G}(v)$ be the set of vertices adjacent to $v$ in $G$, where $d_{G}(v)=|N_{G}(v)|$ is the degree of $v$ in $G$. Denote by $\delta(G)$ and $\Delta(G)$  the  minimum and maximum degrees,  respectively,  of the vertices in $G$. Recalling  the definition of $\tau(G,\omega)$ in \eqref{x1-1}, if  the weights of $G$ are induced by indeterminates $\{x_{i}:v_i\in V(G)\}$, then we denote $\tau(G,\omega)$ by a {\it multivariable polynomial} $P_G(x_1,x_2,\ldots,x_n)$ \cite{L1} associated with $G$ and these indeterminates. That is,
$$P_G(x_1,x_2,\ldots,x_n)=\tau(G,\omega)=\sum_{T\in \mathcal{T}(G)}\prod_{v_iv_j\in E(T)}x_{i}x_{j}=\sum_{T\in \mathcal{T}(G)}\prod_{v_i\in V(G)}x_i^{d_T(v_i)}.$$

The enumeration of spanning trees in a graph is  an important and popular problem in graph theory having connections with many other fields  including  theoretical physics, computer science and so on. There are several methods for counting spanning trees in graphs such as the Matrix-Tree theorem  \cite{B1,K3} and the  electrical network  transformations \cite{T1}. Some scholars  investigated the number of spanning trees in regular graphs \cite{A2,M3,P1}, graphs with symmetry \cite{Y1}, line graphs \cite{D2,G2}, almost complete multipartite graphs \cite{C4}, the $K_n$-complement of  bipartite graphs \cite{G1}, Ferrers graphs \cite{E1} and  graphs with a given maximum degree \cite{D3}. Meanwhile, other scholars studied the number of spanning trees containing a given spanning forest in graphs  \cite{D1,L2,L3,Y2}.  Notably, Klee and Stamps \cite{K1} and Zhou and Bu \cite{Z3} gave some explicit formulas for enumerating spanning trees in  weighted graphs  with weights induced by $\{x_{i}:v_i\in V(G)\}$, which will be fully applied in our work.

A spanning tree $T$ in $G$ is called {\it homeomorphically irreducible spanning tree},  or HIST  for short, if $T$ contains no vertex of degree two. In particular, an {\it  odd spanning tree} in a connected graph  is a spanning tree in which  every vertex  has odd degree. Obviously, an odd spanning tree must be a HIST.   The existence of a HIST in a graph is extensively studied in the literature \cite{A1,C1,H1,I11,S1}.  In particular,  Albertson, Berman, Hutchinson and Thomassen \cite{A1} showed that every connected graph of order $n$ with $\delta(G)\geq \min\{\frac{n}{2},4\sqrt{2n}\}$ contains a HIST. Subsequently, Zheng and Wu \cite{Z1}  established the following result.

\begin{theorem}[\hspace{1sp}{\cite{Z1}}]\label{x1-1}
 Let $n$ be a positive even number. If $G$ is a connected graph of order $n$ with $\delta(G)\geq\frac{n}{2}+1$, then $G$ has an odd spanning tree.
\end{theorem}

According to Theorem \ref{x1-1}, there are  many graphs having at least one odd spanning tree. Inspired by the aforementioned results on spanning tree enumeration, a natural question is to investigate the exact number of  odd spanning trees in certain special graphs. Denote by  $\tau_o(G)$ the number of  odd spanning trees in a  graph  $G$. The well-known Cayley's formula \cite{C5} states that  $\tau(K_n)=n^{n-2}$, where $K_{n}$ is  the complete graph  on $n$ vertices. Recently,   Feng, Chen and Wu \cite{F1} derived an  explicit formula for enumerating   odd spanning trees in  $K_n$, employing the classical Pr\"{u}fer sequence and  an associated generating function.
\begin{theorem}[\hspace{1sp}{\cite{F1}}]\label{x1-2}
 For a complete graph $K_n$, we have
 $$\tau_o(K_n)=
 \begin{cases}
  0  &\mbox {\rm if $n$ is odd},\\
\frac{1}{2^{n}}\sum_{k=0}^{n}\binom{n}{k}(2k-n)^{n-2}&\mbox {\rm if $n$ is even}.\\
   \end{cases}$$
\end{theorem}

In this paper we further study the enumeration of odd spanning trees in graphs. In Section $2$, we review the weighted Matrix-Tree theorem, which establishes a link between the multivariable polynomial $P_G(x_1,x_2,\ldots,x_n)$ of a graph $G$ and  a cofactor of its weighted Laplacian matrix. Using this   polynomial, we give  a unified technique
for counting odd spanning trees in general connected graphs. As applications, the closed enumerative formulas for odd spanning trees in complete graphs, complete multipartite graphs,  almost complete graphs, complete split graphs and Ferrers graphs are, respectively, derived in Section $3$. Finally, we provide two related problems in Section $4$.

\section{Counting odd spanning trees in  general graphs}
Suppose that $G$ is  a weighted graph with $V(G)=\{v_1,v_2,\ldots,v_n\}$ whose weights  are induced by indeterminates $\{x_{i}:v_i\in V(G)\}$.  The {\it weighted Laplacian matrix} $L_{G}$ of $G$ is an $n\times n$ matrix with entries
\begin{displaymath}
(L_{G})_{ij}=
   \begin{cases}
   x_i\sum_{v_k\in N_G(v_i)}x_k &\mbox {\rm if $v_i=v_j$},\\
 -x_ix_j &\mbox {\rm if $v_iv_j\in E(G)$},\\
 0  &\mbox {\rm otherwise}.
   \end{cases}
\end{displaymath}

Let $A(i,j)$ denote the submatrix of a matrix $A$ obtained by deleting the $i$-th row and $j$-th column in $A$. In what follows, we only present the Matrix-Tree theorem for the vertex-weighted version,  which expresses the multivariable polynomial $P_G(x_1,x_2,\ldots,x_n)$ via the determinant of $L_G(i, j)$, see   \cite{B1,L1,K3} for details.

\begin{theorem}[\hspace{1sp}{\cite{B1,L1,K3}}]\label{x2-1}
Let $G$ be a weighted graph with $V(G)=\{v_1,v_2,...,v_n\}$. Suppose that the weights of $G$ are induced by indeterminates $\{x_{i}:v_i\in V(G)\}$. For any (not necessarily
distinct) vertices $v_i,v_j\in V(G)$, we have $\det(L_{G}(i,j))=P_G(x_1,x_2,\ldots,x_n)$.
\end{theorem}

Let $\operatorname{adj}(A)$  be the {\it adjoint matrix} of a square matrix $A$.  We now review a well-known result in linear algebra, which is called the {\it matrix determinant lemma} or {\it Cauchy's formula for the determinant of a rank-one perturbation} \cite{H3}.
\begin{theorem}[\hspace{1sp}{\cite{H3}}]\label{BX}
Let $A$ be an $n\times n$ matrix and let $\mathbf{a}$ and $\mathbf{b}$ be two column vectors in $\mathbb{R}^{n}$. Then
$$\det(A+\mathbf{a}\mathbf{b}^{\top})=\det(A)+\mathbf{b}^{\top}\operatorname{adj}(A)\mathbf{a}.$$
In particular, if $A$ is  nonsingular, then $\det(A+\mathbf{a}\mathbf{b}^{\top})=\det(A)\left(1+\mathbf{b}^{\top}A^{-1}\mathbf{a}\right).$
\end{theorem}

Note that the number of vertices with odd degree is even in any connected graph. Thus, every connected graph of odd order has no odd spanning trees.

\begin{Remark}\label{odd-0}
If $G$ is a connected graph of odd order, then $\tau_o(G)=0$.
\end{Remark}

Let $\{\pm1\}=\{+1,-1\}$ and  $\{\pm1\}^{n}$ be the set of all vectors  of dimension $n$ with each entry independently taking  $+1$ or $-1$. In what follows, we restrict our attention to graphs of even order.   We now state our main result, which gives a unified technique to count odd spanning trees in a general graph.

\begin{theorem}\label{x2-2}
Let $G$ be a connected graph of even order  $n$ and  $P_G(x_1,x_2,\ldots,x_n)$ be a multivariable polynomial associated with $G$ and  indeterminates  $\{x_i:v_i\in V(G)\}$. Then
$$\tau_o(G)=\frac{1}{2^{n}}\sum_{\boldsymbol{\sigma}\in\{\pm1\}^{n}}\left(\prod_{i=1}^{n}\sigma_i\right)P_G(\sigma_1,\sigma_2,\ldots,\sigma_n),$$
where the sum runs over all assignment vectors $\boldsymbol{\sigma}=(\sigma_1,\sigma_2,\ldots,\sigma_n)$ with each $\sigma_i$ independently taking  $+1$ or $-1$.
\end{theorem}
\begin{proof}
Consider the expansion of the polynomial $P_G(x_1,x_2,\ldots,x_n)$:
\begin{equation}\label{xh-4}
P_G(x_1,x_2,\ldots,x_n)=\sum_{T\in \mathcal{T}(G)}\prod_{v_i\in V(G)}x_i^{d_T(v_i)}=\sum_\mathbf{d}c_{\mathbf{d}}x_1^{d_1}x_2^{d_2}\cdots x_n^{d_n},
\end{equation}
where $\mathbf{d}=(d_1,d_2,\ldots,d_n)$ runs over all possible degree sequences, and the coefficient $c_\mathbf{d}$ is the number of spanning trees in $G$ whose degree sequence is $\mathbf{d}$. Observe that $\tau_o(G)$  is exactly the sum of those coefficients $c_{\mathbf{d}}$  for which every $d_i$ is odd. That is,
\begin{equation}\label{xh3-2}
\tau_o(G)=\sum_{\substack{\mathbf{d}=(d_1,d_2,\ldots,d_n),\\\mathrm{each}~d_i~\mathrm{is~odd}}}c_{\mathbf{d}}.
\end{equation}

Now we define
$$R=\frac{1}{2^n}\sum_{\boldsymbol{\sigma}\in\{\pm 1\}^n}\left(\prod_{i=1}^n\sigma_i\right)P_G(\sigma_1,\ldots,\sigma_n).$$
Substituting $P_G(\sigma_1,\sigma_2,\ldots,\sigma_n)$  with its expansion from Equation \eqref{xh-4} gives
\begin{equation}\label{xhf-1}
\begin{split}
R&=\frac{1}{2^n}\sum_{\boldsymbol{\sigma}\in\{\pm 1\}^n}\left(\prod_{i=1}^n\sigma_i\right)\sum_{\mathbf{d}}c_{\mathbf{d}}\prod_{i=1}^n\sigma_i^{d_i}\\
&=\frac{1}{2^n}\sum_{\mathbf{d}}c_{\mathbf{d}}\sum_{ \boldsymbol{\sigma}\in\{\pm 1\}^n}\prod_{i=1}^n\sigma_i^{d_i+1}\\
&=\frac{1}{2^n}\sum_{\mathbf{d}}c_{\mathbf{d}}\sum_{\sigma_1\in\{\pm 1\}}\sum_{\sigma_2\in\{\pm 1\}}\cdots\sum_{\sigma_n\in\{\pm 1\}}\prod_{i=1}^n\sigma_i^{d_i+1}\\
&=\frac{1}{2^n}\sum_{\mathbf{d}}c_{\mathbf{d}}\prod_{i=1}^n\left(\sum_{\sigma_i\in\{\pm 1\}}\sigma_i^{d_i+1}\right).
\end{split}
\end{equation}
For a fixed \(\mathbf{d}\), we compute each factor $\sum_{\sigma_i\in\{\pm 1\}}\sigma_i^{d_i+1}$. Consequently,
\begin{itemize}
    \item If $d_i$ is odd, then $d_i+1$ is even, so $\sigma_i^{d_i+1}=1$. Hence,
    $\sum_{\sigma_i\in\{\pm 1\}}\sigma_i^{d_i+1}=1+1=2$;
    \item If $d_i$ is even, then $d_i+1$ is odd, so $\sigma_i^{d_i+1}=\sigma_i$. Hence,
    $\sum_{\sigma_i\in\{\pm 1\}}\sigma_i^{d_i+1}=1+(-1)=0$.
\end{itemize}
Therefore,
\begin{equation}\label{xh-3}
\prod_{i=1}^n\left(\sum_{\sigma_i\in \{\pm 1\}}\sigma_i^{d_i+1}\right)=
\begin{cases}
2^n, & \text{if each }d_i\text{ is odd},\\[2mm]
0,   & \text{otherwise}.
\end{cases}
\end{equation}
Substituting with Equation \eqref{xh-3} the corresponding part in the expression of $R$ in \eqref{xhf-1} and combining it with Equation \eqref{xh3-2}, we obtain
$$
R=\frac{1}{2^n}\sum_{\substack{\mathbf{d}=(d_1,d_2,\ldots,d_n),\\\mathrm{each}~d_i~\mathrm{is~odd}}}2^{n}c_{\mathbf{d}}
= \sum_{\substack{\mathbf{d}=(d_1,d_2,\ldots,d_n),\\\mathrm{each}~d_i~\mathrm{is~odd}}}c_{\mathbf{d}}=\tau_{o}(G),
$$
the result follows  as desired.
\end{proof}

\section{Applications}
 In this section, by applying Theorem \ref{x2-2}, we obtain  explicit formulas for the number of odd spanning trees in some special classes of  graphs such as complete graphs, complete multipartite graphs,  almost complete graphs, complete split graphs and Ferrers graphs.

\subsection{Complete graphs}

Recall that the polynomial $P_G(x_1,x_2,\ldots,x_n)$ denotes the sum of weights of spanning trees in a graph $G$ whose weights  are induced by indeterminates $\{x_{i}:v_i\in V(G)\}$.  Now we state a result from \cite[Theorem 2]{K1}.
\begin{theorem}[\hspace{1sp}{\cite{K1}}]\label{x3-1}
   Suppose that  the weights of the complete graph $K_n$ are induced by indeterminates $\{x_{i}:v_i\in V(K_n)\}$. Then
$$P_{K_n}(x_1,x_2,\ldots,x_n)=x_1x_2\cdots x_n(x_1+x_2+\cdots+x_n)^{n-2}.$$
\end{theorem}

Next we provide  a new  proof of Theorem \ref{x1-2}.
\begin{proof}[\em\textbf {Proof of Theorem \ref{x1-2}}]
By  Theorem \ref{x3-1}, for any assignment vector $\boldsymbol{\sigma}=(\sigma_1,\sigma_2,\ldots,\sigma_n)\in \{\pm 1\}^{n}$ with even number $n$, we compute
\begin{equation}\label{x3-2}
\begin{split}
\left(\prod_{i=1}^{n}\sigma_i\right)P_{K_n}(\sigma_1,\sigma_2,\ldots,\sigma_n)=\left(\prod_{i=1}^{n}\sigma_i\right)^{2}\left(\sum_{i=1}^{n}\sigma_i\right)^{n-2}=\left(\sum_{i=1}^{n}\sigma_i\right)^{n-2}.
\end{split}
\end{equation}

Let $k$ be the number of $+1$s in $\boldsymbol{\sigma}$. Then the number of $-1$s in  $\boldsymbol{\sigma}$  is $n-k$. In this case,
\begin{equation}\label{x3-3}
\sum_{i=1}^{n}\sigma_i=k\cdot(+1)+(n-k)\cdot(-1)=2k-n.
\end{equation}
For a fixed nonnegative integer $k$ with $0\leq k\leq n$, there are $\binom{n}{k}$ such assignment vectors $\boldsymbol{\sigma}$  with exactly $k$ entries of $+1$. Consequently, by Theorem \ref{x2-2} and combining with Equations \eqref{x3-2} and  \eqref{x3-3}, we have
\begin{equation*}
\begin{split}
\tau_o(K_n)&=\frac{1}{2^{n}}\sum_{\boldsymbol{\sigma}\in\{\pm1\}^{n}}\left(\prod_{i=1}^{n}\sigma_i\right)P_{K_n}(\sigma_1,\sigma_2,\ldots,\sigma_n)\\
&=\frac{1}{2^{n}}\sum_{\boldsymbol{\sigma}\in\{\pm1\}^{n}}\left(\sum_{i=1}^{n}\sigma_i\right)^{n-2}\\
&=\frac{1}{2^{n}}\sum_{k=0}^{n}\binom{n}{k}(2k-n)^{n-2}.
\end{split}
\end{equation*}
The result follows  from the above with Remark \ref{odd-0}.
\end{proof}

\subsection{Complete multipartite graphs}

Let $G=K_{n_1,n_2,\ldots,n_k}$ be a {\it complete multipartite graph}  with  vertex set $V(G)$ partitioned as  $V_{1}\cup V_{2}\cup \cdots\cup V_{k}$ where  $|V_{i}|=n_{i}$ for $1\leq i\leq k$. Let $\sum_{i=1}^{k}n_i=n$  and order the vertices of $G$ so that $v_1,v_2,\ldots, v_{n_1}\in V_1$, $v_{n_1+1},v_{n_1+2},\ldots, v_{n_1+n_2}\in V_2$ and so on. Klee and  Stamps \cite[Theorem 3]{K1} gave a new proof of the following theorem.

\begin{theorem}[\hspace{1sp}{\cite{K1}}]\label{xsh-1}
   Let $G=K_{n_1,n_2,\ldots,n_k}$ be the complete multipartite graph with $n$ vertices and $k$ parts described above. Assume that  the weights of $G$ are induced by indeterminates $\{x_{i}:v_i\in V(G)\}$. Then
   $$P_G(x_1,x_2,\ldots,x_n)=\left(\prod_{i=1}^{n}x_i\right)\left(\prod_{\ell=1}^{k}\Big(\sum_{v_j\not\in V_\ell}x_j\Big)^{n_\ell-1}\right)\left(\sum_{i=1}^{n}x_i\right)^{k-2}.$$
\end{theorem}

Now we give the explicit formula for the number of odd spanning trees in a complete multipartite graph.
\begin{theorem}\label{xshxsh-1}
 For a complete multipartite  graph $G=K_{n_1,n_2,\ldots,n_k}$ with even $n=\sum_{i=1}^{k}n_i$ vertices and $k$ parts described above, we have
 $$\tau_o(G)=\frac{1}{2^{n}}\sum_{r_1=0}^{n_1}\sum_{r_2=0}^{n_2}\cdots\sum_{r_k=0}^{n_k}\left(\prod_{\ell=1}^{k}\binom{n_\ell}{r_\ell}\right)\left(\prod_{\ell=1}^{k}(M-m_\ell)^{n_{\ell}-1}\right)M^{k-2},$$
where $m_\ell=2r_\ell-n_\ell$ and $M=\sum_{\ell=1}^{k}m_\ell$ for  $1\leq \ell\leq k$.
\end{theorem}
\begin{proof}
 By Theorem \ref{xsh-1}, for any assignment vector $\boldsymbol{\sigma}=(\sigma_1,\sigma_2,\ldots,\sigma_n)\in \{\pm 1\}^{n}$, we have
\begin{equation}\label{xsh-2}
\begin{split}
\left(\prod_{i=1}^{n}\sigma_i\right)P_{G}(\sigma_1,\sigma_2,\ldots,\sigma_n)&=\left(\prod_{i=1}^{n}\sigma_i\right)^{2}\left(\prod_{\ell=1}^{k}\Big(\sum_{v_j\not\in V_\ell}\sigma_j\Big)^{n_\ell-1}\right)\left(\sum_{i=1}^{n}\sigma_i\right)^{k-2}\\
&=\left(\prod_{\ell=1}^{k}\Big(\sum_{v_j\not\in V_\ell}\sigma_j\Big)^{n_\ell-1}\right)\left(\sum_{i=1}^{n}\sigma_i\right)^{k-2}.
\end{split}
\end{equation}

 Now we classify the assignment vector $\boldsymbol{\sigma} \in \{\pm 1\}^n$. Let each  $V_\ell$ have $r_\ell$ ($0\leq r_\ell\leq n_\ell$) vertices  with assigned $+1$, where $1\leq \ell\leq k$. Then the number of vertices in  $V_\ell$  with assigned $-1$ is $n_\ell-r_\ell$. Moreover, we define $m_\ell=2r_\ell-n_\ell $ and $M=\sum_{\ell=1}^{k}m_\ell$ for  $1\leq \ell\leq k$. Then
\begin{equation}\label{xsh-3}
\sum_{i=1}^{n}\sigma_i=\sum_{\ell=1}^{k}\left(r_\ell\cdot (+1)+(n_\ell-r_\ell)\cdot(-1)\right)=\sum_{\ell=1}^{k}\left(2r_\ell-n_\ell\right)=\sum_{\ell=1}^{k}m_\ell=M;
\end{equation}
\begin{equation}\label{xsh-4}
\sum_{v_j\not\in V_\ell}\sigma_j=\sum_{i=1}^{n}\sigma_i- \sum_{v_j\in V_\ell}\sigma_j  =M-m_\ell.
\end{equation}

In a fixed $k$-tuple $(r_1,r_2,\ldots,r_k)$ with  $0\leq r_\ell\leq n_\ell$ for each $\ell$, there are  $\prod_{\ell=1}^{k}\binom{n_\ell}{r_\ell}$   such assignment vectors $\boldsymbol{\sigma}$ that have exactly $r_\ell$ vertices assigned $+1$ in each $V_\ell$. By Theorem \ref{x2-2} and  Equations \eqref{xsh-2},  \eqref{xsh-3} and  \eqref{xsh-4},  we have
\begin{equation*}
\begin{split}
\tau_o(G)&=\frac{1}{2^{n}}\sum_{\boldsymbol{\sigma}\in\{\pm1\}^{n}}\left(\prod_{i=1}^{n}\sigma_i\right)P_G(\sigma_1,\sigma_2,\ldots,\sigma_n)\\
&=\frac{1}{2^{n}}\sum_{\boldsymbol{\sigma}\in\{\pm1\}^{n}}\left(\prod_{\ell=1}^{k}\Big(\sum_{v_j\not\in V_\ell}\sigma_j\Big)^{n_\ell-1}\right)\left(\sum_{i=1}^{n}\sigma_i\right)^{k-2}\\
&=\frac{1}{2^{n}}\sum_{\substack{(r_1,r_2,\ldots,r_k),\\\mathrm 0\leq r_\ell\leq n_\ell}}
\left(\prod_{\ell=1}^{k}\binom{n_\ell}{r_\ell}\right)\left(\prod_{\ell=1}^{k}(M-m_\ell)^{n_{\ell}-1}\right)M^{k-2}\\
&=\frac{1}{2^{n}}\sum_{r_1=0}^{n_1}\sum_{r_2=0}^{n_2}\cdots\sum_{r_k=0}^{n_k}\left(\prod_{\ell=1}^{k}\binom{n_\ell}{r_\ell}\right)\left(\prod_{\ell=1}^{k}(M-m_\ell)^{n_{\ell}-1}\right)M^{k-2},
\end{split}
\end{equation*}
 which completes the proof.
\end{proof}

Note that there are exactly two non-isomorphic unlabeled odd trees  of order $6$ as follows.
\begin{enumerate}[label=\text{(\alph*).}]
\item The  star $K_{1,5}$: A central vertex (degree $5$) connected to $5$ leaves;
\item The double star $DS(2,2)$: Two adjacent central vertices (both of degree $3$), each connected to $2$ leaves.
\end{enumerate}

Next, we provide a simple example for Theorem \ref{xshxsh-1}.
\begin{example}
For the complete tripartite graph $K_{2,2,2}$, there is no odd spanning tree  isomorphic to $K_{1,5}$ since $\Delta(K_{2,2,2})=4$. It is easy to see that there are $24$ odd spanning trees isomorphic to $DS(2,2)$ in $K_{2,2,2}$.

Moreover, applying Theorem  $\ref{xshxsh-1}$ yields
$$\tau_o(K_{2,2,2})=\frac{1}{2^{6}}\sum_{r_{1}=0}^{2}\sum_{r_{2}=0}^{2}\sum_{r_{3}=0}^{2}\left(\prod_{\ell=1}^{3}\binom{2}{r_\ell}\right) \left(\prod_{\ell=1}^{3}(M-m_\ell)\right)M=\frac{1}{64}\cdot 1536=24,$$
 where $m_\ell=2r_\ell-n_\ell$ and $M=\sum_{\ell=1}^{3}m_\ell$ for  $1\leq \ell\leq 3$. Clearly, this agrees with the result obtained by direct computation.
\end{example}

\subsection{Almost complete graphs}

A graph is called {\it almost complete}, denoted by $K_n-pK_2$, if it is obtained from the complete graph $K_n$ by deleting a matching of $p$ edges where $2p\leq n$. Zhou and Bu \cite[Example 3.11]{Z3} gave a  formula  for counting spanning trees
in $K_n-pK_2$.

\begin{theorem}[\hspace{1sp}{\cite{Z3}}]\label{xshz-1}
 Let $G=K_n-pK_2$ be obtained from $K_n$ by deleting a matching $\mathcal{M}$ of size $p$ with weights induced by indeterminates $\{x_{i}:v_i\in V(G)\}$. Then
$$P_G(x_1,x_2,\ldots,x_n)=\left(\prod_{i=1}^{n}x_i\right)\left(\sum_{i=1}^{n}x_i\right)^{n-p-2}\prod_{v_iv_j\in \mathcal{M}}\sum_{k\neq i,j}x_k.$$
\end{theorem}

Now we give the explicit formula for the number of odd spanning trees in an almost complete  graph.
\begin{theorem}\label{xshz-2}
 For an almost complete  graph $G=K_n-pK_2$  of even order $n$ obtained from $K_n$ by deleting a matching $\mathcal{M}$ of size $p$, we have
 $$
\tau_o(G)= \frac{1}{2^n} \sum_{\substack{a+b+c=p \\ 0 \le r \le n-2p}} \frac{p!\,2^{b}}{a!\,b!\,c!} \, \binom{n-2p}{r} \, S^{\,n-p-2} \, (S-2)^a \, S^b \, (S+2)^c,
$$
where  $S = 4a + 2b + 2r - n$ and $a,b,c,r$ are nonnegative integers.
\end{theorem}
\begin{proof}
 By Theorems \ref{x2-2} and \ref{xshz-1}, for any assignment vector $\boldsymbol{\sigma}=(\sigma_1,\sigma_2,\ldots,\sigma_n)\in \{\pm 1\}^{n}$, we have
 \begin{equation}\label{xshzj-3}
\begin{split}
\tau_o(G)&=\frac{1}{2^{n}}\sum_{\boldsymbol{\sigma}\in\{\pm1\}^{n}}\left(\prod_{i=1}^{n}\sigma_i\right)P_G(\sigma_1,\sigma_2,\ldots,\sigma_n)\\
&=\frac{1}{2^{n}}\sum_{\boldsymbol{\sigma}\in\{\pm1\}^{n}}\left(\prod_{i=1}^{n}\sigma_i\right)^{2}\left(\sum_{i=1}^{n}\sigma_i\right)^{n-p-2}\prod_{v_iv_j\in \mathcal{M}}\sum_{k\neq i,j}\sigma_k\\
&=\frac{1}{2^{n}}\sum_{\boldsymbol{\sigma}\in\{\pm1\}^{n}}S^{n-p-2}\prod_{v_iv_j\in \mathcal{M}}(S-\sigma_i-\sigma_j),
\end{split}
\end{equation}
where $S=\sum_{i=1}^{n}\sigma_i$.

 Now we classify the assignment vector $\boldsymbol{\sigma} \in \{\pm 1\}^n$. Recall that the matching $\mathcal{M}$ contains $p$ edges.  For each edge $v_iv_j \in \mathcal{M}$, there are three cases  for the pair $(\sigma_i, \sigma_j)$ as follows.
\begin{enumerate}[label=\text{(\roman*).},leftmargin=4em]
\item $(\sigma_i, \sigma_j)=(+1, +1)$:  in which case $S - \sigma_i - \sigma_j = S - 2$;
\item $(\sigma_i, \sigma_j)=(+1, -1)$ or $(-1, +1)$: in which case $S - \sigma_i - \sigma_j = S$;
\item $(\sigma_i, \sigma_j)=(-1, -1)$: in which case $S - \sigma_i - \sigma_j = S+2$.
\end{enumerate}
Let $a, b, c$ denote the number of edges in the three cases described above, respectively,  with $a + b + c = p$.  Moreover, let $r$ be the number of vertices not in $\mathcal{M}$ (there are $n - 2p$  vertices  not in $\mathcal{M}$) that are assigned the  value $+1$, with the rest assigned $-1$, where $0\leq r\leq n - 2p$. Then,  for a fixed quadruple $(a,b,c,r)$,
  \begin{equation}\label{xshz-4}
\prod_{v_iv_j\in \mathcal{M}}(S-\sigma_i-\sigma_j)=(S-2)^{a}S^{b}(S+2)^{c};
\end{equation}
\begin{equation}\label{xshz-5}
S=\sum_{i=1}^{n}\sigma_i=2a-2c+2r-n+2p=4a+2b+2r-n.
\end{equation}
Furthermore, the number of such assignment vectors $\boldsymbol{\sigma}$ with a given  quadruple $(a,b,c,r)$ is computed as follows:
\begin{itemize}
\item Select $a$ edges from the matching $\mathcal{M}$ to assign values of type (\romannumeral 1\relax), $b$ edges to assign  values of type (\romannumeral 2\relax) and $c$ edges to assign values of type (\romannumeral 3\relax). The number of such ways is given by the multinomial coefficient $\binom{p}{a,\,b,\,c}= \frac{p!}{a!\,b!\,c!}$;
\item  For the $b$ edges assigned different values, there are two ways to specify which endpoint is $+1$ for each edge, contributing a factor of $2^{b}$;
\item Choose $r$ vertices from the $n-2p$ vertices not in $\mathcal{M}$ to assign $+1$. The number of such ways is  $\binom{n-2p}{r}$.
\end{itemize}
Therefore, the total count of such assignment vectors $\boldsymbol{\sigma}$ is
\begin{equation}\label{xshz-6}
\frac{p!}{a!\,b!\,c!} \cdot 2^b \cdot \binom{n-2p}{r}.
\end{equation}

Substituting  with Equations \eqref{xshz-4}, \eqref{xshz-5} and \eqref{xshz-6}  into \eqref{xshzj-3} and summing over all possible quadruples $(a, b, c, r)$, we obtain
$$
\tau_o(G)= \frac{1}{2^n} \sum_{\substack{a+b+c=p\\ 0 \le r \le n-2p}}\, \frac{p!}{a!\,b!\,c!} \, 2^b \, \binom{n-2p}{r} \, S^{\,n-p-2} \, (S-2)^a \, S^b \, (S+2)^c,
$$
with $S = 4a + 2b + 2r - n$. This completes the proof.
\end{proof}

Next, we illustrate Theorem \ref{xshz-2} with a simple example.

\begin{example}
For the almost complete  graph $K_{6}-2K_2$,  there are $2$ odd spanning trees  isomorphic to $K_{1,5}$ and $38$ odd spanning trees  isomorphic to $DS(2,2)$.

Moreover, applying Theorem  $\ref{xshz-2}$ yields
$$\tau_o(K_{6}-2K_2)=
\frac{1}{2^{6}}\sum_{\substack{a+b+c=2\\ 0\leq r\leq 2}}
\frac{2!\,2^{b}}{a!\,b!\,c!}\,\binom{2}{r}
S^{2}(S-2)^{a}S^{b}(S+2)^{c}=\frac{1}{64}\cdot 2560=40,
$$
 where $S=4a+2b+2r-6$.
\end{example}

\subsection{Complete split graphs}

The {\it join} $G_{1}\vee G_{2}$ of two disjoint graphs $G_{1}$ and $G_{2}$ is the graph obtained from  $G_{1}$ and $G_{2}$ by joining each vertex in $G_1$ to each vertex in $G_2$. Let $\overline{G}$ denote the {\it complement} of a graph  $G$. A {\it split graph} is a graph whose vertices can be partitioned into a clique and an independent set. Denote by $CS(m,n)=K_m\vee \overline{K_n}$ the {\it complete split graph}  with bipartition $V(CS(m,n))=V(K_m)\cup V(\overline{K_n})$. We  give a formula for counting spanning trees in a weighted complete split graph as follows.

\begin{theorem}\label{xshd-1}
   Let $G=CS(m,n)$ be a complete split graph  with bipartition $V(K_m)\cup V(\overline{K_n})$, where $V(K_m)=\{u_1,u_2,\ldots,u_m\}$ and $V(\overline{K_n})=\{v_1,v_2,\ldots,v_n\}$. If  the weights of $G$ are induced by indeterminates $\{x_{i}:u_i\in V(K_m)\}\cup \{y_{j}:v_j\in V(\overline{K_n})\}$, then
   $$P_{G}(x_1,\ldots,x_m,y_1,\ldots,y_n)=\left( \prod_{i=1}^m x_i \right)\left( \prod_{j=1}^n y_j \right)\left(\sum_{i=1}^{m}x_i\right)^{n-1}\left(\sum_{i=1}^{m}x_i+\sum_{j=1}^{n}y_j\right)^{m-1}.$$
\end{theorem}
\begin{proof}
Let $\mathbf{x}=(x_1,x_2,\ldots,x_m)^{\top}$ and $\mathbf{y}=(y_1,y_2,\ldots,y_n)^{\top}$ be column vectors of dimension $m$ and $n$, respectively.  The weighted Laplacian matrix $L_G$ can be partitioned into the following block matrix corresponding to $V(K_m)$ and $V(\overline{K_n})$\,:
$$L_{G}=\left(\begin{matrix}
A &B\\
B^{\top} &D\\
\end{matrix}
\right),$$
where $A=(S_\mathbf{x}+S_\mathbf{y})\cdot \operatorname{diag}(x_1,x_2,\ldots,x_m)-\mathbf{x}\mathbf{x}^{\top}$, $D=S_\mathbf{x}\cdot \operatorname{diag}(y_1,y_2,\ldots,y_n)$ is a diagonal matrix and $B=-\mathbf{x}\mathbf{y}^{\top}$, with $S_\mathbf{x}=\sum_{i=1}^{m}x_i$ and $S_\mathbf{y}=\sum_{j=1}^{n}y_j$.

Deleting the row and column in $L_G$ corresponding to the last  vertex in  $\overline{K_n}$, we obtain the cofactor matrix $L_G(n,n)$. Let $\mathbf{y}' = (y_1, y_2,\ldots, y_{n-1})^\top$ and $S_{\mathbf{y}'}= \sum_{j=1}^{n-1} y_j$. Then
$$L_{G}(n,n)=\left(\begin{matrix}
A &B'\\
B'^{\top} &D'\\
\end{matrix}
\right),$$
 with $B'=-\mathbf{x}\mathbf{y}'^{\top}$ and $D'=S_\mathbf{x}\cdot\operatorname{diag}(y_1,y_2,\ldots,y_{n-1})$.

First, we assume that $D'$ is nonsingular. Using the block matrix determinant formula, we have
\begin{equation}\label{xshd-2}
\det(L_G(n,n)) = \det(D') \det(A - B' D'^{-1} B'^{\top}).
\end{equation}
By performing a routine computation, we have
\begin{equation*}
\begin{split}
B' D'^{-1}=\left(-\mathbf{x}\mathbf{y}'^{\top}\right)\frac{1}{S_\mathbf{x}}\cdot \operatorname{diag}(\frac{1}{y_1},\frac{1}{y_2},\ldots,\frac{1}{y_{n-1}})=-\frac{1}{S_\mathbf{x}}\mathbf{x}\mathbf{1}_{n-1}^{\top},
\end{split}
\end{equation*}
where   $\mathbf{1}_{n-1}$ denotes the  all-ones column vector of dimension $n-1$. Moreover,
\begin{equation*}
\begin{split}
B'D'^{-1}B'^\top=\left(-\frac{1}{S_\mathbf{x}}\mathbf{x}\mathbf{1}_{n-1}^{\top}\right)\left(-\mathbf{y}'\mathbf{x}^{\top}\right)=\frac{1}{S_\mathbf{x}}\mathbf{x}\left(\mathbf{1}_{n-1}^{\top}\mathbf{y}'\right)\mathbf{x}^{\top}=\frac{S_{\mathbf{y}'}}{S_\mathbf{x}}\mathbf{x}\mathbf{x}^{\top}.
 \end{split}
\end{equation*}
Consequently,
\begin{equation*}
\begin{split}
A - B' D'^{-1} B'^\top&=(S_\mathbf{x}+S_\mathbf{y})\cdot \operatorname{diag}(x_1,x_2,\ldots,x_m)-\frac{S_\mathbf{x}+S_{\mathbf{y}'}}{S_\mathbf{x}}\mathbf{x}\mathbf{x}^{\top}.
 \end{split}
\end{equation*}
Assume that $A_0=(S_\mathbf{x}+S_\mathbf{y})\cdot \operatorname{diag}(x_1,x_2,\ldots,x_m)$  is nonsingular.   By the matrix determinant lemma, we have
\begin{equation}\label{xshd-3}
\begin{split}
\det(A - B' D'^{-1} B'^{\top})&=\det(A_0)\left(1-\frac{S_\mathbf{x}+S_{\mathbf{y}'}}{S_\mathbf{x}}\mathbf{x}^{\top}A_0^{-1}\mathbf{x}\right)\\
&=(S_\mathbf{x}+S_\mathbf{y})^{m}\left(\prod_{i=1}^{m}x_i\right)\left(1-\frac{S_\mathbf{x}+S_{\mathbf{y}'}}{S_\mathbf{x}}\frac{S_\mathbf{x}}{S_\mathbf{x}+S_{\mathbf{y}}}\right)\\
&=(S_\mathbf{x}+S_\mathbf{y})^{m}\left(\prod_{i=1}^{m}x_i\right)\frac{y_n}{S_\mathbf{x}+S_{\mathbf{y}}}\\
&=y_n(S_\mathbf{x}+S_\mathbf{y})^{m-1}\prod_{i=1}^{m}x_i.
 \end{split}
\end{equation}

By Theorem \ref{x2-1} and combining with  Equations \eqref{xshd-2} and \eqref{xshd-3}, we have
\begin{equation}\label{xshd-4}
\begin{split}
P_{G}(x_1,\ldots,x_m,y_1,\ldots,y_n)=&\det(L_G(n,n))=\det(D') \det(A - B' D'^{-1} B'^{\top})\\
=&\left(S_\mathbf{x}^{n-1}\prod_{j=1}^{n-1}y_j\right)\left(y_n(S_\mathbf{x}+S_\mathbf{y})^{m-1}\prod_{i=1}^{m}x_i\right)\\
=&\left( \prod_{i=1}^m x_i \cdot\prod_{j=1}^n y_j \right)\left(\sum_{i=1}^{m}x_i\right)^{n-1}\left(\sum_{i=1}^{m}x_i+\sum_{j=1}^{n}y_j\right)^{m-1}.
 \end{split}
\end{equation}

In the discussion above, we assume that both $D'$ and $A_0$ are nonsingular. Note that  $P_G(x_1,\ldots,x_m,y_1,\ldots,y_n)$ is a polynomial function in variables $\{x_i\}_{1\leq i\leq m}$ and $\{y_j\}_{1\leq j\leq n}$,  and the result holds when  both $D'$ and $A_0$ are nonsingular. For the case where either  $D'$ or $A_0$ is singular, we have  four cases: $S_\mathbf{x}=0$, $S_\mathbf{x}+S_\mathbf{y}=0$, $x_i=0$ or $y_j=0$ for $1\leq i\leq m$ and $1\leq j\leq n$. Thus, it follows from the continuity of polynomial functions that  Equation \eqref{xshd-4}  is valid for all real assignments of the weights.
\end{proof}

Now we give the explicit formula for the number of odd spanning trees in a  complete split  graph.
\begin{theorem}\label{xshu-1}
  Let  $CS(m,n)$ be a complete split  graph  of even order $m+n$ with bipartition $V(K_m)\cup V(\overline{K_n})$, where $V(K_m)=\{u_1,u_2,\ldots,u_m\}$ and $V(\overline{K_n})=\{v_1,v_2,\ldots,v_n\}$. Then
 $$\tau_o(CS(m,n))=\frac{1}{2^{m+n}}\sum_{k=0}^{m}\sum_{\ell=0}^{n}\binom{m}{k}\binom{n}{\ell}\left(2k-m\right)^{n-1}\left( 2k+2\ell-m-n\right)^{m-1}.$$
\end{theorem}
\begin{proof}
For any assignment vector $\boldsymbol{\sigma}=(\sigma_1,\sigma_2,\ldots,\sigma_{m+n})\in \{\pm 1\}^{m+n}$, we decompose it as $\sigma_i=\alpha_i$ for $1\leq i\leq m$ and $\sigma_{m+j}=\beta_j$ for $1\leq j\leq n$ such that $\boldsymbol{\alpha}=(\alpha_1,\alpha_2,\ldots,\alpha_{m})\in \{\pm 1\}^{m}$  and $\boldsymbol{\beta}=(\beta_1,\beta_2,\ldots,\beta_{n})\in \{\pm 1\}^{n}$ represent the weights assigned to the vertices in  $V(K_m)$ and $V(\overline{K_n})$, respectively. Letting $G=CS(m,n)$ and by Theorems \ref{x2-2} and \ref{xshd-1},  we have
\begin{equation}\label{xshu-2}
\begin{split}
\tau_o(G)&=\frac{1}{2^{m+n}}\sum_{\boldsymbol{\sigma}\in\{\pm1\}^{m+n}}\left(\prod_{i=1}^{m+n}\sigma_i\right)P_G(\sigma_1,\sigma_2,\ldots,\sigma_{m+n})\\
&=\frac{1}{2^{m+n}}\sum_{\boldsymbol{\sigma}\in\{\pm1\}^{m+n}}\left(\prod_{i=1}^{m}\alpha_i\cdot\prod_{j=1}^{n}\beta_j\right)^{2}\left(\sum_{i=1}^{m}\alpha_i\right)^{n-1}\left(\sum_{i=1}^{m}\alpha_i+\sum_{j=1}^{n}\beta_j\right)^{m-1}\\
&=\frac{1}{2^{m+n}}\sum_{\boldsymbol{\alpha}\in\{\pm1\}^{m},\boldsymbol{\beta}\in\{\pm1\}^{n}} \left(\sum_{i=1}^{m}\alpha_i\right)^{n-1}\left(\sum_{i=1}^{m}\alpha_i+\sum_{j=1}^{n}\beta_j\right)^{m-1}.\\
\end{split}
\end{equation}

Now we classify the assignment vector $\boldsymbol{\sigma} \in \{\pm 1\}^{m+n}$. Let $k$ and $\ell$ be the number of vertices in  $K_m$ and $\overline{K_n}$  that are assigned the value $+1$, respectively. In this case,
$$\sum_{i=1}^{m}\alpha_i=k\cdot (+1)+(m-k)\cdot(-1)=2k-m;$$
$$\sum_{i=1}^{m}\alpha_i+\sum_{j=1}^{n}\beta_j=(2k-m)+\ell\cdot (+1)+(n-\ell)\cdot(-1)=2k+2\ell-m-n.$$
 Moreover, for  fixed nonnegative integers $k$ and $\ell$ with $0\leq k\leq m$ and $0\leq \ell\leq n$, there are $\binom{m}{k}\binom{n}{\ell}$  assignment vectors $\boldsymbol{\sigma}$ that have exactly $k$ and $\ell$ vertices  with assigned $+1$ in  $K_m$ and $\overline{K_n}$, respectively. Thus, by   Equation \eqref{xshu-2}, we have
\begin{equation*}
\begin{split}
\tau_o(G)&=\frac{1}{2^{m+n}}\sum_{\boldsymbol{\alpha}\in\{\pm1\}^{m},\boldsymbol{\beta}\in\{\pm1\}^{n}} \left(\sum_{i=1}^{m}\alpha_i\right)^{n-1}\left(\sum_{i=1}^{m}\alpha_i+\sum_{j=1}^{n}\beta_j\right)^{m-1}\\
&=\frac{1}{2^{m+n}}\sum_{k=0}^{m}\sum_{\ell=0}^{n}\binom{m}{k}\binom{n}{\ell}\left(2k-m\right)^{n-1}\left( 2k+2\ell-m-n\right)^{m-1}.
\end{split}
\end{equation*}
This completes the proof.
\end{proof}

Next, we provide a simple example for Theorem \ref{xshu-1}.
\begin{example}
For the complete split graph $CS(3,3)$, there are $3$ odd spanning trees  isomorphic to $K_{1,5}$ and $27$ odd spanning trees  isomorphic to $DS(2,2)$.

Moreover, applying Theorem  $\ref{xshu-1}$ yields
$$\tau_o(CS(3,3))=\frac{1}{2^{6}}\sum_{k=0}^{3}\sum_{\ell=0}^{3}\binom{3}{k}\binom{3}{\ell}\left(2k-3\right)^{2}\left( 2k+2\ell-6\right)^{2}=\frac{1}{64}\cdot 1920=30.
$$
\end{example}

\subsection{Ferrers graphs}
In 2004, Ehrenborg and van Willigenburg \cite{E1} defined a class of bipartite graphs called Ferrers graphs. They used the theory of electrical networks to   give  a precise formula \cite[Theorem 2.1]{E1} for counting spanning trees in a weighted Ferrers graph.
\begin{definition}
A Ferrers graph is a connected bipartite graph $G$ whose vertices can be partitioned as $R\cup C$ with $R=\{r_1,r_2,\ldots,r_m\}$ and $C=\{c_1,c_2,\ldots,c_n\}$ such that
\begin{itemize}
\item $r_1c_{n}$ and $r_m c_1$ are edges in $G$, and
\item whenever $r_pc_q$ is an edge in $G$, then so is $r_ic_j$ for any $1\leq i \leq p$ and  $1\leq j \leq q$.
\end{itemize}
\end{definition}

A {\it partition} of a positive integer $z$ is an ordered list of positive integers whose sum is $z$.  We write  $\lambda = (\lambda_1,\lambda_2, \ldots, \lambda_m)$ with $\lambda_1\geq \lambda_2\geq\cdots\geq \lambda_m>0$ to denote the parts of the partition $\lambda$. For a Ferrers graph $G$ we have the associated partition $\lambda=(\lambda_1,\lambda_2, \ldots, \lambda_m)$, where $\lambda_i$ is the degree of the vertex $r_i$ for each $1\leq i\leq m$. Similarly, we have the {\it dual partition} $\lambda'=(\lambda_1',\lambda_2',\ldots,\lambda_n')$, where
$\lambda_j'$ is the degree of the vertex $c_j$ for each $1\leq j\leq n$.  The associated {\it Ferrers diagram} is  a diagram of boxes  with a box at position $(r_i,c_j)$ if and only if $r_ic_j\in E(G)$. That is, it is
a stack of left-justified boxes with $\lambda_1$ boxes in the first row, $\lambda_2$ boxes in the second row, and so on. Therefore, there exists a natural correspondence between Ferrers graphs and Ferrers diagrams, which also represent integer partitions. For example, $(5,5,3,2,1)$ is a partition of $16$, the Ferrers diagram and corresponding Ferrers graph associated to the partition $(5,5,3,2,1)$ are shown in Fig.\,$1$.

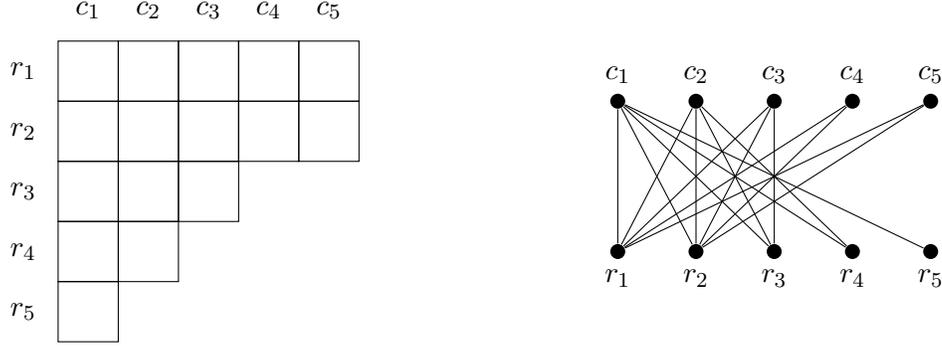
\begin{figure}[h]
\centering
\begin{tikzpicture}[scale=0.8]
\draw (0,0) rectangle (1,1);
\draw (1,0) rectangle (2,1);
\draw (2,0) rectangle (3,1);
\draw (3,0) rectangle (4,1);
\draw (4,0) rectangle (5,1);

\draw (0,-1) rectangle (1,0);
\draw (1,-1) rectangle (2,0);
\draw (2,-1) rectangle (3,0);
\draw (3,-1) rectangle (4,0);
\draw (4,-1) rectangle (5,0);

\draw (0,-2) rectangle (1,-1);
\draw (1,-2) rectangle (2,-1);
\draw (2,-2) rectangle (3,-1);

\draw (0,-3) rectangle (1,-2);
\draw (1,-3) rectangle (2,-2);

\draw (0,-4) rectangle (1,-3);

\node[left] at (-0.2, 0.5) {$r_1$};
\node[left] at (-0.2, -0.5) {$r_2$};
\node[left] at (-0.2, -1.5) {$r_3$};
\node[left] at (-0.2, -2.5) {$r_4$};
\node[left] at (-0.2, -3.5) {$r_5$};

\node[above] at (0.5, 1.2) {$c_1$};
\node[above] at (1.5, 1.2) {$c_2$};
\node[above] at (2.5, 1.2) {$c_3$};
\node[above] at (3.5, 1.2) {$c_4$};
\node[above] at (4.5, 1.2) {$c_5$};

\pgfmathsetmacro{\shift}{8}

\foreach \i/\label in {1/c_1,2/c_2,3/c_3,4/c_4,5/c_5} {
    \node[circle,fill=black,inner sep=2pt] (c\i) at (\shift+\i*1.3,0) {};
    \node[above] at (c\i.north) {$\label$};
}

\foreach \i/\label in {1/r_1,2/r_2,3/r_3,4/r_4,5/r_5} {
    \node[circle,fill=black,inner sep=2pt] (r\i) at (\shift+\i*1.3,-2.5) {};
    \node[below] at (r\i.south) {$\label$};
}

\foreach \i in {1,...,5} {
    \draw (c1) -- (r\i);
}

\foreach \i in {1,...,4} {
    \draw (c2) -- (r\i);
}

\foreach \i in {1,...,3} {
    \draw (c3) -- (r\i);
}

\foreach \i in {1,2} {
    \draw (c4) -- (r\i);
}

\foreach \i in {1,2} {
    \draw (c5) -- (r\i);
}

\end{tikzpicture}
\caption*{\footnotesize\textbf{Fig.\,1.} The Ferrers diagram (left) and Ferrers graph (right) corresponding to the partition $(5,5,3,2,1)$.}
\end{figure}

\begin{theorem}[\hspace{1sp}{\cite{K1,E1}}]\label{thm:ferrers}
Let $G$ be a Ferrers graph whose vertices are partitioned as $V(G) = R \cup C$ with $|R| = m$ and $|C| = n$. Suppose that $G$ corresponds to the partition $\lambda = (\lambda_1,\lambda_2, \ldots, \lambda_m)$ whose dual partition is $\lambda' = (\lambda'_1, \lambda'_2, \ldots, \lambda'_n)$. If  the weights of $G$ are induced by indeterminates $\{x_{i}:r_i\in R\}\cup \{y_{j}:c_j\in C\}$, then
\begin{equation*}
\begin{split}
 &P_G(x_1,\ldots,x_m,y_1,\ldots,y_n)
 = \left( \prod_{i=1}^m x_i\right) \left(\prod_{j=1}^n y_j \right) \left( \prod_{i=2}^m \Big( \sum_{j=1}^{\lambda_i} y_j \Big) \right)  \left( \prod_{j=2}^n \Big( \sum_{i=1}^{\lambda'_j} x_i \Big) \right).
 \end{split}
\end{equation*}
\end{theorem}

By applying this  enumerative   formula, we present a formula for counting odd spanning trees in a Ferrers graph.

\begin{theorem}\label{xshzy-1}
 Let $G$ be a Ferrers graph whose vertices are partitioned as $V(G) = R \cup C$ with $|R| = m$ and $|C| = n$, where $m+n$ is even. Suppose that $G$ corresponds to the partition $\lambda = (\lambda_1,\lambda_2, \ldots, \lambda_m)$ whose dual partition is $\lambda' = (\lambda'_1, \lambda'_2, \ldots, \lambda'_n)$. Then
 \begin{equation*}
\begin{split}
 \tau_o(G)=&\frac{1}{2^{m+n}}\left(\sum_{k_{1}=0}^{p_{1}}\sum_{k_{2}=0}^{p_{2}}\cdots\sum_{k_{m}=0}^{p_{m}}\Big(\prod_{t=1}^{m}\binom{p_{t}}{k_{t}}\Big)\prod_{i=2}^{m}\Big(2\sum_{t=i}^{m}k_t-\lambda_{i}\Big)\right)\\
 &\cdot\left(\sum_{\ell_{1}=0}^{q_{1}}\sum_{\ell_{2}=0}^{q_{2}}\cdots\sum_{\ell_{n}=0}^{q_{n}}\Big(\prod_{s=1}^{n}\binom{q_{s}}{\ell_{s}}\Big)\prod_{j=2}^{n}\Big(2\sum_{s=j}^{n}\ell_s-\lambda_{j}^{\prime}\Big)\right),
\end{split}
\end{equation*}
where $p_i=\lambda_i-\lambda_{i+1}$ with $\lambda_{m+1}=0$  and $q_j=\lambda_j'-\lambda_{j+1}'$ with $\lambda_{n+1}'=0$.
\end{theorem}
\begin{proof}
By Theorems \ref{x2-2} and \ref{thm:ferrers}, for any assignment vector $\sigma=(\sigma_1,\sigma_2,\ldots,\sigma_{m+n})\in \{\pm 1\}^{m+n}$ with $\sigma_i=\alpha_i\in \{\pm 1\}$ for $1\leq i\leq m$ and $\sigma_{m+j}=\beta_j\in \{\pm 1\}$ for $1\leq j\leq n$, we compute
 \begin{equation}\label{xshz-3}
\begin{split}
\tau_o(G)&=\frac{1}{2^{m+n}}\sum_{\boldsymbol{\sigma}\in\{\pm1\}^{m+n}}\left(\prod_{i=1}^{m+n}\sigma_i\right)P_G(\sigma_1,\sigma_2,\ldots,\sigma_{m+n})\\
&=\frac{1}{2^{m+n}}\sum_{\boldsymbol{\sigma}\in\{\pm1\}^{m+n}}\left(\prod_{i=1}^{m}\alpha_i\cdot\prod_{j=1}^{n}\beta_j\right)^{2}\left( \prod_{i=2}^m \Big( \sum_{j=1}^{\lambda_i} \beta_j \Big) \right)  \left( \prod_{j=2}^n \Big( \sum_{i=1}^{\lambda'_j} \alpha_i \Big) \right)\\
&=\frac{1}{2^{m+n}}\sum_{\boldsymbol{\sigma}\in\{\pm1\}^{m+n}}\left( \prod_{i=2}^m \Big( \sum_{j=1}^{\lambda_i} \beta_j \Big) \right)  \left( \prod_{j=2}^n \Big( \sum_{i=1}^{\lambda'_j} \alpha_i \Big) \right)\\
&=\frac{1}{2^{m+n}}\left(\sum_{\boldsymbol{\beta}\in\{\pm1\}^{n}}\prod_{i=2}^m \Big( \sum_{j=1}^{\lambda_i} \beta_j \Big) \right) \left(\sum_{\boldsymbol{\alpha}\in\{\pm1\}^{m}}\prod_{j=2}^n \Big( \sum_{i=1}^{\lambda'_j} \alpha_i\Big)\right),
\end{split}
\end{equation}
where $\boldsymbol{\alpha}=(\alpha_1,\alpha_2,\ldots,\alpha_{m})\in \{\pm 1\}^{m}$  and $\boldsymbol{\beta}=(\beta_1,\beta_2,\ldots,\beta_{n})\in \{\pm 1\}^{n}$, which represent the weights assigned to the vertices in  $R$ and $C$, respectively.

Now we classify the  assignment  vectors $\boldsymbol{\alpha} \in \{\pm 1\}^{m}$ and $\boldsymbol{\beta} \in \{\pm 1\}^{n}$. Since $G$ is a Ferrers graph, we have $N_G(r_1)\supseteq N_G(r_2)\supseteq \cdots \supseteq N_G(r_m)$.  Let $r_{m+1}$ and $c_{n+1}$ be two virtual isolated vertices in $G$, that is, $N_G(r_{m+1})=\emptyset$ and $N_G(c_{n+1})=\emptyset$. Then we partition the vertex set $C$ into $m$ disjoint sets $C_1$, $C_2$,\ldots, $C_m$, where $C_i=N_G(r_i)- N_G(r_{i+1})$ and $|C_i|=p_i=\lambda_i-\lambda_{i+1}$  for $i\in \{1,2,\ldots,m\}$. Similarly, we partition the vertex set $R$ into $n$ disjoint sets $R_1$, $R_2$,\ldots, $R_n$, where  $R_j=N_G(c_j)- N_G(c_{j+1})$ and $|R_j|=q_j=\lambda_j'-\lambda_{j+1}'$  for $j\in \{1,2,\ldots,n\}$.  For $1\leq i\leq m$ and $1\leq j\leq n$, let $k_i$ and $\ell_j$ be the number of vertices in  $C_i$ and $R_j$  that are assigned the value $+1$, respectively. Observe that  $0\leq k_i\leq p_i$ and  $0\leq \ell_j\leq q_j$. Moreover,
 $$ \sum_{j=1}^{\lambda_i} \beta_j=\sum_{t=i}^{m}(2k_t-p_t)=2\sum_{t=i}^{m}k_t-\sum_{t=i}^{m}p_t=2\sum_{t=i}^{m}k_t-\lambda_i,$$
 $$ \sum_{i=1}^{\lambda_j'} \alpha_i=\sum_{s=j}^{n}(2\ell_s-q_s)=2\sum_{s=j}^{n}\ell_s-\sum_{s=j}^{n}q_s=2\sum_{s=j}^{n}\ell_s-\lambda_j'.$$

 For fixed nonnegative integers $k_1,k_2,\ldots,k_m$ with $0\leq k_i\leq p_i$  for $1\leq i\leq m$, the number of  such assignment vectors $\boldsymbol{\alpha}$ corresponding to these parameters is $\prod_{t=1}^{m}\binom{p_t}{k_t}$. Similarly, for  fixed nonnegative integers $\ell_1,\ell_2,\ldots,\ell_n$ with $0\leq \ell_j\leq q_j$  for $1\leq j\leq n$, the number of  such corresponding assignment vectors $\boldsymbol{\beta}$ is $\prod_{s=1}^{n}\binom{q_s}{\ell_s}$. By  Equation \eqref{xshz-3}, we have
 \begin{equation*}
\begin{split}
\tau_o(G)=&\frac{1}{2^{m+n}}\left(\sum_{\boldsymbol{\beta}\in\{\pm1\}^{n}}\prod_{i=2}^m \Big( \sum_{j=1}^{\lambda_i} \beta_j \Big) \right) \left(\sum_{\boldsymbol{\alpha}\in\{\pm1\}^{m}}\prod_{j=2}^n \Big( \sum_{i=1}^{\lambda'_j} \alpha_i\Big)\right)\\
=&\frac{1}{2^{m+n}}\left(\sum_{k_{1}=0}^{p_{1}}\sum_{k_{2}=0}^{p_{2}}\cdots\sum_{k_{m}=0}^{p_{m}}\Big(\prod_{t=1}^{m}\binom{p_{t}}{k_{t}}\Big)\prod_{i=2}^{m}\Big(2\sum_{t=i}^{m}k_t-\lambda_{i}\Big)\right)\\
 &\cdot\left(\sum_{\ell_{1}=0}^{q_{1}}\sum_{\ell_{2}=0}^{q_{2}}\cdots\sum_{\ell_{n}=0}^{q_{n}}\Big(\prod_{s=1}^{n}\binom{q_{s}}{\ell_{s}}\Big)\prod_{j=2}^{n}\Big(2\sum_{s=j}^{n}\ell_s-\lambda_{j}^{\prime}\Big)\right),
\end{split}
\end{equation*}
 which leads to the result.
\end{proof}

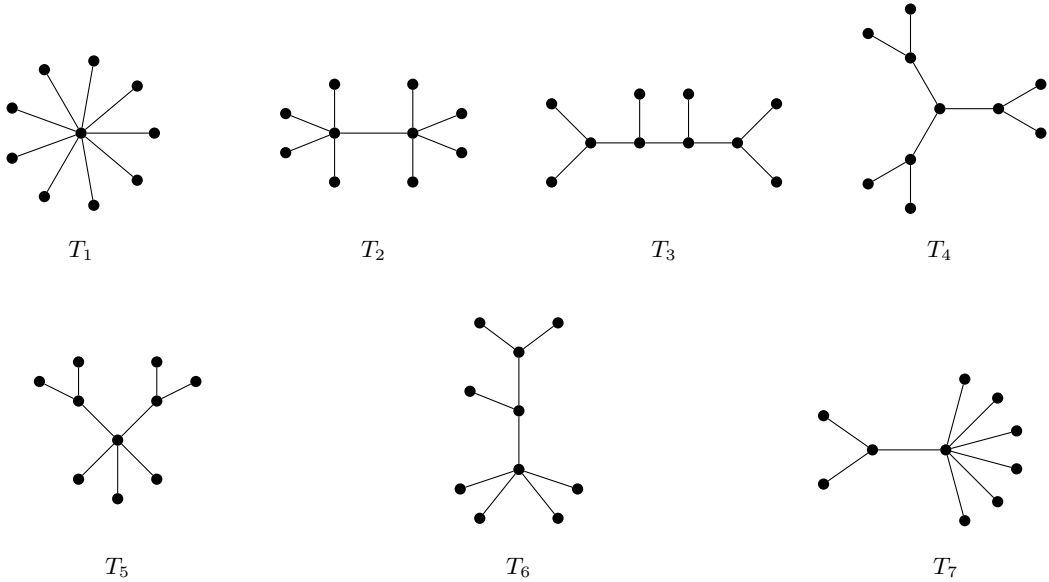
\begin{figure}[htbp]
\centering
\begin{minipage}[b]{0.24\textwidth}
\centering
\begin{tikzpicture}[scale=0.65]
  \tikzset{dot/.style={circle,fill,inner sep=1.5pt,outer sep=0pt}}
  \node[dot] (center) at (0,0) {};
  \foreach \angle in {0,40,80,120,160,200,240,280,320} {
    \node[dot] (leaf) at (\angle:1.5) {};
    \draw (center) -- (leaf);
  }
  \node[below, font=\footnotesize] at (0,-2) {$T_1$};
\end{tikzpicture}
\end{minipage}
\hfill
\begin{minipage}[b]{0.24\textwidth}
\centering
\begin{tikzpicture}[scale=0.65]
  \tikzset{dot/.style={circle,fill,inner sep=1.5pt,outer sep=0pt}}
  \node[dot] (v5a) at (-0.8,0) {};
  \node[dot] (v5b) at (0.8,0) {};
  \draw (v5a) -- (v5b);
  \node[dot] (la1) at (-0.8,1) {};
  \node[dot] (la2) at (-0.8,-1) {};
  \node[dot] (la3) at (-1.8,0.4) {};
  \node[dot] (la4) at (-1.8,-0.4) {};
  \draw (v5a) -- (la1) (v5a) -- (la2) (v5a) -- (la3) (v5a) -- (la4);
  \node[dot] (lb1) at (0.8,1) {};
  \node[dot] (lb2) at (0.8,-1) {};
  \node[dot] (lb3) at (1.8,0.4) {};
  \node[dot] (lb4) at (1.8,-0.4) {};
  \draw (v5b) -- (lb1) (v5b) -- (lb2) (v5b) -- (lb3) (v5b) -- (lb4);
  \node[below, font=\footnotesize] at (0,-2) {$T_2$};
\end{tikzpicture}
\end{minipage}
\hfill
\begin{minipage}[b]{0.24\textwidth}
\centering
\begin{tikzpicture}[scale=0.65]
  \tikzset{dot/.style={circle,fill,inner sep=1.5pt,outer sep=0pt}}
  \node[dot] (v1) at (-1.5,0) {};
  \node[dot] (v2) at (-0.5,0) {};
  \node[dot] (v3) at (0.5,0) {};
  \node[dot] (v4) at (1.5,0) {};
  \draw (v1) -- (v2) -- (v3) -- (v4);
  \node[dot] (l11) at (-2.3,0.8) {};
  \node[dot] (l12) at (-2.3,-0.8) {};
  \draw (v1) -- (l11) (v1) -- (l12);
  \node[dot] (l21) at (-0.5,1) {};
  \draw (v2) -- (l21);
  \node[dot] (l31) at (0.5,1) {};
  \draw (v3) -- (l31);
  \node[dot] (l41) at (2.3,0.8) {};
  \node[dot] (l42) at (2.3,-0.8) {};
  \draw (v4) -- (l41) (v4) -- (l42);
  \node[below, font=\footnotesize] at (0,-1.8) {$T_3$};
\end{tikzpicture}
\end{minipage}
\hfill
\begin{minipage}[b]{0.24\textwidth}
\centering
\begin{tikzpicture}[scale=0.65]
  \tikzset{dot/.style={circle,fill,inner sep=1.5pt,outer sep=0pt}}

  \node[dot] (center) at (0,0) {};

  \foreach \angle in {0,120,240} {
    \node[dot] (branch) at (\angle:1.2) {};
    \draw (center) -- (branch);

    \node[dot] (leaf1) at ($(branch)+(\angle+30:1)$) {};
    \node[dot] (leaf2) at ($(branch)+(\angle-30:1)$) {};
    \draw (branch) -- (leaf1) (branch) -- (leaf2);
  }
  \node[below, font=\footnotesize] at (0,-2.5) {$T_4$};
\end{tikzpicture}
\end{minipage}

\vspace{0.6cm}

\begin{minipage}[b]{0.30\textwidth}
\centering
\begin{tikzpicture}[scale=0.65]
  \tikzset{dot/.style={circle,fill,inner sep=1.5pt,outer sep=0pt}}

  \node[dot] (v5) at (0,0) {};

  \node[dot] (l51) at (-0.8,-0.8) {};
  \node[dot] (l52) at (0,-1.2) {};
  \node[dot] (l53) at (0.8,-0.8) {};
  \draw (v5) -- (l51) (v5) -- (l52) (v5) -- (l53);

  \node[dot] (v3a) at (-0.8,0.8) {};
  \node[dot] (v3b) at (0.8,0.8) {};
  \draw (v5) -- (v3a) (v5) -- (v3b);

  \node[dot] (l3a1) at (-1.6,1.2) {};
  \node[dot] (l3a2) at (-0.8,1.6) {};
  \draw (v3a) -- (l3a1) (v3a) -- (l3a2);

  \node[dot] (l3b1) at (0.8,1.6) {};
  \node[dot] (l3b2) at (1.6,1.2) {};
  \draw (v3b) -- (l3b1) (v3b) -- (l3b2);

  \node[below, font=\footnotesize] at (0,-2.2) {$T_5$};
\end{tikzpicture}
\end{minipage}
\hfill
\begin{minipage}[b]{0.30\textwidth}
\centering
\begin{tikzpicture}[scale=0.65]
  \tikzset{dot/.style={circle,fill,inner sep=1.5pt,outer sep=0pt}}
  \node[dot] (v5) at (0,-0.6) {};
  \node[dot] (v3a) at (0,0.6) {};
  \node[dot] (v3b) at (0,1.8) {};
  \draw (v5) -- (v3a) -- (v3b);
  \node[dot] (l51) at (-1.2,-1) {};
  \node[dot] (l52) at (-0.8,-1.6) {};
  \node[dot] (l53) at (0.8,-1.6) {};
  \node[dot] (l54) at (1.2,-1) {};
  \draw (v5) -- (l51) (v5) -- (l52) (v5) -- (l53) (v5) -- (l54);
  \node[dot] (l3a) at (-1,1) {};
  \draw (v3a) -- (l3a);
  \node[dot] (l3b1) at (-0.8,2.4) {};
  \node[dot] (l3b2) at (0.8,2.4) {};
  \draw (v3b) -- (l3b1) (v3b) -- (l3b2);
  \node[below, font=\footnotesize] at (0,-2.2) {$T_6$};
\end{tikzpicture}
\end{minipage}
\hfill
\begin{minipage}[b]{0.30\textwidth}
\centering
\begin{tikzpicture}[scale=0.65]
  \tikzset{dot/.style={circle,fill,inner sep=1.5pt,outer sep=0pt}}
  \node[dot] (center) at (0,0) {};
  \foreach \angle in {-75,-45,-15,15,45,75} {
    \node[dot] (leaf) at (\angle:1.5) {};
    \draw (center) -- (leaf);
  }
  \node[dot] (branch) at (-1.5,0) {};
  \draw (center) -- (branch);
  \node[dot] (leaf1) at (-2.5,0.7) {};
  \node[dot] (leaf2) at (-2.5,-0.7) {};
  \draw (branch) -- (leaf1) (branch) -- (leaf2);
  \node[below, font=\footnotesize] at (0,-2) {$T_7$};
\end{tikzpicture}
\end{minipage}
\caption*{\footnotesize\textbf{Fig.\,2.} All non-isomorphic unlabeled odd  trees of order 10: $T_i$ for $i\in \{1,2,\ldots,7\}$.}
\end{figure}

Note that there are exactly $7$ non-isomorphic unlabeled odd trees  of order $10$, shown in Fig.\,$2$. Finally, we provide a simple example for Theorem \ref{xshzy-1}.
\begin{example}
As shown in Fig.\,$1$, let $G$ be the Ferrers  graph   corresponding to the partition $(5,5,3,2,1)$ with dual partition $(5,4, 3, 2,2)$. Observe that there is no odd spanning tree in $G$ isomorphic to $T_1$, $T_4$, $T_6$ or $T_7$. Note that there are $2$ odd spanning trees isomorphic to $T_2$, $32$  ones  isomorphic to $T_3$,  and  $16$  ones isomorphic to $T_5$.

Moreover, applying Theorem  $\ref{xshzy-1}$, we have $p_1=0$, $p_2=2$, $p_3=1$, $p_4=1$, $p_5=1$ and $q_1=1$, $q_2=1$, $q_3=1$, $q_4=0$, $q_5=2$. Then $k_1=0$ and $\ell_4=0$. Consequently,
\begin{equation*}
\begin{split}
\tau_o(G)=&\frac{1}{2^{10}}\left(\sum_{k_{2}=0}^{2}\sum_{k_{3}=0}^{1}\sum_{k_{4}=0}^{1}\sum_{k_{5}=0}^{1}\binom{2}{k_2}\binom{1}{k_3}\binom{1}{k_4}\binom{1}{k_5}\prod_{i=2}^{5}\Big(2\sum_{t=i}^{5}k_t-\lambda_{i}\Big)\right)\\
&\cdot\left(\sum_{\ell_{1}=0}^{1}\sum_{\ell_{2}=0}^{1}\sum_{\ell_{3}=0}^{1}\sum_{\ell_{5}=0}^{2}\binom{1}{\ell_1}\binom{1}{\ell_2}\binom{1}{\ell_3}\binom{2}{\ell_5}\prod_{j=2}^{5}\Big(2\sum_{s=j}^{5}\ell_s-\lambda_{j}'\Big)\right)\\
=&\frac{1}{2^{10}}\cdot160\cdot320=50.
\end{split}
\end{equation*}

\end{example}

\section{Concluding remarks}
 In this paper we  provide a  unified technique to  enumerate odd spanning trees in general graphs. As applications, the  explicit formulas for counting odd spanning trees in several special classes of  graphs are derived from our work. Recall that an odd spanning tree must be a HIST, which is a spanning tree  containing no vertex of degree two. This leads to the following relevant problems worthy of investigation.

\begin{problem}
Find the formulas for enumerating homeomorphically irreducible spanning trees in   some special classes of  graphs.
\end{problem}

\begin{problem}
Find the formulas for enumerating  spanning trees with other degree restriction conditions in   some special classes of  graphs.
\end{problem}

\section*{Data Availability Statement}
No data is used in this research.

\section*{Conflict of Interest Statement}
The authors have no conflict of interest.

\section*{Acknowledgements}
The work was supported by National Natural Science Foundation of China (Grant No.\ 12271251), Postgraduate Research \& Practice Innovation Program of Jiangsu Province, grant number KYCX25\_0625.


\end{document}